\documentclass[11pt]{amsart}
\usepackage{a4wide}
\usepackage{version}
\usepackage{amsmath}
\usepackage{amsfonts, amssymb, amsthm}
\usepackage{pstricks}
\usepackage{mhequ}
\usepackage{color,mathrsfs,epsfig}
\usepackage{epsfig}
\usepackage{esint}
\usepackage{color}
\usepackage{hyperref}
\hypersetup{
    colorlinks = true,
    linkcolor = blue,
    anchorcolor = blue,
    citecolor = blue,
    filecolor = blue,
    urlcolor = blue
    }

\DeclareMathAlphabet{\mathpzc}{OT1}{pzc}{m}{it}
\newtheorem{thm}{Theorem}

\newtheorem{preremark}[thm]{Remark}
\newenvironment{remark}{\begin{preremark}\rm}{\medskip \end{preremark}}
\numberwithin{equation}{section}
\newtheorem{prop}[thm]{Proposition}
\newtheorem{lemma}[thm]{Lemma}
\newtheorem{defn}[thm]{Definition}

\numberwithin{equation}{section}

\newcommand{\R}{\mathbb R}

\newcommand{\eps}{\varepsilon}

\newcommand{\dx} {\; \mathrm{d} x}
\newcommand{\dy} {\; \mathrm{d} y}

\DeclareMathOperator{\supp}{supp}
\DeclareMathOperator{\cov}{cov}

\newcommand{\meanbar}[1]{%
\setbox0 = \hbox{$#1 \int$}
\hbox to 0pt{%
\thinspace
\hskip 0.1\wd0
\raise 0.5\ht0
\hbox{%
\lower 0.5\dp0
\hbox{\rule{0.8\wd0}{2\linethickness}}
}%
\hss
}%
}

\makeatletter
\def\@tocline#1#2#3#4#5#6#7{\relax
  \ifnum #1>\c@tocdepth 
  \else
    \par \addpenalty\@secpenalty\addvspace{#2}%
    \begingroup \hyphenpenalty\@M
    \@ifempty{#4}{%
      \@tempdima\csname r@tocindent\number#1\endcsname\relax
    }{%
      \@tempdima#4\relax
    }%
    \parindent\z@ \leftskip#3\relax \advance\leftskip\@tempdima\relax
    \rightskip\@pnumwidth plus4em \parfillskip-\@pnumwidth
    #5\leavevmode\hskip-\@tempdima
      \ifcase #1
       \or\or \hskip 1em \or \hskip 2em \else \hskip 3em \fi%
      #6\nobreak\relax
    \hfill\hbox to\@pnumwidth{\@tocpagenum{#7}}\par
    \nobreak
    \endgroup
  \fi}
\makeatother

\author[Georgiana Chatzigeorgiou]{Georgiana Chatzigeorgiou$^{1}$}  

\title{Locality properties of standard homogenization commutator}

\begin{document}

\begin{abstract} 
In the present work we study how the \textit{standard homogenization commutator}, a random field that plays a central role in the theory of fluctuations, quantitatively decorrelates on large scales.

\end{abstract}

\let\thefootnote\relax\footnotetext{ $^{1}$Max Planck Institute for Mathematics in the Sciences, Inselstrasse 22, 04103 Leipzig, Germany

\tt $\ $georgiana.chatzigeorgiou@mis.mpg.de
}


\maketitle



\section{Introduction}

This work amounts to homogenization theory for uniformly elliptic linear equations in divergence-form. That is, we consider
\begin{equs} \label{eq}
\nabla  \cdot \left(a \left( \frac{\cdot}{\eps} \right) \nabla u_{\eps}+ f \right)=0, \quad \text{ in } \ \R^{d}
\end{equs}
with $f \in C^{\infty}_{c}(\R^{d})^{d}$ and $a$ being random (not necessarily symmetric) coefficients  that satisfy
\begin{equs}
\xi \cdot a(x) \xi \geq \lambda |\xi|^{2}, \quad  \xi \cdot a^{-1}(x) \xi \geq  |\xi|^{2} \quad \text{ for any } \ x,\xi \in \R^{d}
\end{equs}
for some positive constant $\lambda$. In the following we denote by $\langle \cdot \rangle$ the expectation with respect to the underlying measure on $a$'s. 

Since the works of Papanicolaou and Varadhan \cite{PV81} and Kozlov \cite{K79} we know that in stationary and ergodic random environments, the equation (\ref{eq}) homogenizes as $\eps \to 0$ to an equation
\begin{equs} \label{eq-hom}
\nabla  \cdot \left(\bar{a}  \nabla \bar{u}+ f \right)=0, \quad \text{ in } \ \R^{d}
\end{equs}
where the coefficients $\bar{a}$ are constant and deterministic. More precisely, the effective coefficients $\bar{a}$ are given by $\bar{a} e_{i} := \big \langle a(\nabla \phi_{i}+e_{i}) \big \rangle$ where the corrector $\phi_{i}$ is the (up to a random additive constant) unique a.s. solution of the equation $\nabla \cdot a(\nabla \phi_{i}+e_{i})=0$ in $\R^{d}$,
with $\nabla \phi_{i}$ being stationary, centered and having finite second moments. Furthermore, the aforementioned qualitative theory states that $\nabla u_{\eps}$ weakly converges to $\nabla \bar{u}$ with the oscillations of $u_{\eps}$ being captured by those of the so-called two-scale expansion $\left(1+\eps \phi_{i} \left( \frac{\cdot}{\eps} \right) \partial_{i}\right)\bar{u}$ in a strong norm. 

The quantitative theory of stochastic homogenization  for (\ref{eq}) (namely, the study of the error for the approximation $\nabla u_{\eps} \approx \nabla ( (1+\eps \phi_{i} \left( \frac{\cdot}{\eps} \right) \partial_{i})\bar{u} )$), has been also well-developed during the last decade. For that a suitable quantification of the ergodicity assumption is needed. Most of the developments are based on either a spectral gap inequality or on a finite range of dependence assumption. Here we adopt the spectral gap inequality approach which means that we have a version of Poincar\'e's inequality in infinite dimensions. Roughly speaking, we assume that the variance of an observable defined on the space of coefficient fields described above, can be estimated by a suitable norm of its functional derivative with respect to $a$, which describes the sensitivity of an observable (for instance $\nabla u_{\eps}$ or $\nabla \phi_{i}$) under changes on $a$'s. This direction of research was initiated by Gloria and Otto in \cite{GO11} and \cite{GO12}, inspired by the strategy introduced by Nadaff and Spencer  in \cite{NS98} . On the other hand, finite range of dependence and mixing conditions have been introduced by Yurinski\u{\i} in \cite{Y86} and further studied by Armstrong and Smart in \cite{AS16} (we refer the reader to \cite{AKM} for a detailed description of the progress in this direction).

Next to the spatial oscillations of $\nabla u_{\eps}$ described above, stochastic homogenization  also studies the random fluctuations of observables of the form $\int g \cdot \nabla u_{\eps}$. One of the first results in this direction is given in \cite{GM16} where the authors show that $\eps^{-d/2} \int g \cdot \big ( \nabla u_{\eps} - \big \langle  \nabla u_{\eps} \big \rangle \big )$ converges in law to a Gaussian random variable. In the same work the authors showed that the four-tensor Q introduced in \cite{MO16} describes explicitly the leading-order of the variance of $\eps^{-d/2} \int g \cdot \nabla u_{\eps}$. Moreover, they observed that  the limiting variance of $\eps^{-d/2}\int g \cdot \nabla u_{\eps}$ is not captured by that of $\eps^{-d/2} \int g \cdot \nabla \left( \left(1+\eps \phi_{i} \left( \frac{\cdot}{\eps} \right) \partial_{i}\right)\bar{u} \right)$ as one would naturally expect. As discovered in \cite{DGO20} (for the random conductance model) and \cite{DO20} (in the continuum Gaussian setting) a reasonable quantity to look at, when it comes to fluctuations, is the \textit{homogenization commutator}
\begin{equs}
\Xi^{1}_{\eps}[\nabla u_{\eps}]:=\left( a\left( \frac{\cdot}{\eps} \right) - \bar{a} \right) \nabla u_{\eps}.
\end{equs}
This notion first introduced in \cite{AKM17} and it is highly related to H-convergence which is in fact equivalent to $\Xi^{1}_{\eps} \rightharpoonup 0$ as $\eps \to 0$. The motivation to consider $\Xi^{1}_{\eps}$ while studying fluctuations of $\int g \cdot \nabla u_{\eps}$ comes from the following observation: for the Lax-Milgram solution $\bar{v}$ to the dual equation $\nabla \cdot (\bar{a}^{*}\nabla \bar{v} + g)=0$ we have
\begin{equs}
\int g \cdot \nabla u_{\eps} = -\int \nabla \bar{v} \cdot \bar{a}\nabla u_{\eps} =\int \nabla \bar{v} \cdot \left( a\left( \frac{\cdot}{\eps} \right) - \bar{a} \right)\nabla u_{\eps} + \int \nabla \bar{v} \cdot f.
\end{equs}
That is, the quantity of interest can be written in terms of homogenization commutator plus a deterministic term (which does not contribute to fluctuations). Subsequently, in \cite{DGO20} and \cite{DO20} the authors turned their attention to the study of $\Xi^{1}_{\eps}$ to realize that its fluctuations are captured by those of its two-scale expansion. More precisely, for
$\tilde{F}_{\eps}:= \int g \cdot \left[ \left(a\left( \frac{\cdot}{\eps} \right) - \bar{a} \right) \nabla u_{\eps} - \left( a\left( \frac{\cdot}{\eps} \right) - \bar{a} \right) \left( e_{i} + \nabla \phi_{i} \left(\frac{\cdot}{\eps }\right) \right) \partial_{i} \bar{u} \right]$, it holds (for $d \geq 3$)
\begin{equs}
\eps^{-d/2} \big \langle \big ( \tilde{F}_{\eps} - \langle \tilde{F}_{\eps} \rangle \big )^{2}\big \rangle^{1/2} \lesssim \eps. 
\end{equs}
This reveals the special role that the so-called \textit{standard homogenization commutator}, \begin{equs}
\Xi^{o,1}_{\eps}:=\left( a\left( \frac{\cdot}{\eps} \right) - \bar{a} \right)  \left( Id + \nabla \phi \left(\frac{\cdot}{\eps }\right)\right),
\end{equs}
plays in the theory of fluctuations. In \cite{DGO20} and \cite{DO20} the limiting covariance structure of $\int g \cdot \Xi^{o,1}_{\eps}e_{i}$ is quantitatively characterized and a (quantitative) CLT-type result is obtained for that quantity.

In the present work our aim is to explain how $\Xi^{o,1}_{\eps}$ decorrelates when averaged over balls which are far enough. From now on we set, for convenience, $\eps=1$ and work with macroscopic observables. More precisely, we consider the following macroscopic test functions with supports that are quantitatively ''far'',
\begin{equs} \label{test-fcts}
g(x) =R^ {-d} \eta \left(\frac{x}{R} \right) \in C^\infty_c(B_R) \quad \text{ and } \quad g'(x) =g(x-Le) \in C^\infty_c(B_R(Le)),
\end{equs}
where $\eta \in C^\infty_c(B_1)$ with $|D^{k} \eta| \leq C_{k,d}$, $1\ll R \ll L<\infty$ and $e \in \R^d$ with $|e|=1$. Our aim is to examine how
\begin{equs} \label{correlation}
\Big \langle \int_{\R^d} g(x) \Xi_{ij}^{o,1}(x) \dx \int_{\R^d} g'(y) \Xi_{ml}^{o,1}(y) \dy \Big \rangle
\end{equs}
decays in terms of both $R$ (which amounts to the macroscopic scale) and $L$ (which amounts to the distance between the supports), for every $1 \leq i,j \leq d$. Here, instead of working with the variance first and then appeal to a polarization argument, which would give no information on $L$, we work with the covariance using estimate (\ref{cov_est}) which is an immediate consequence of Hellfer-Sj\"ostrand representation formula. Namely, this work is a refinement of the analysis in \cite{DGO20} and \cite{DO20} on the locality properties of $\Xi^{o,1}$.

As first observed in \cite{DGO20}, one of the main features of $\Xi^{o,1}$ is the approximately local behaviour of its functional derivative in terms of the coefficient field $a$. This property is seen here for the derivative of $F_{ij}^{o,1}:=\int_{\R^d} g(x) \Xi_{ij}^{o,1}(x) \dx$. Precisely, in Proposition \ref{repr_form_prop}, we derive 
\begin{equs} \label{repr_form_1}
  \frac{\partial F_{ij}^{o,1}}{\partial a} = (\nabla \phi _{i}^{1}+e_{i}) \otimes  \left(  (\nabla \phi_{j}^{*} +e_{j} )g+ \phi_{j}^{*} \nabla g) + \nabla h_{j} \right) 
 \end{equs}
 with $- \nabla \cdot a^{*} \nabla h_{j} = \nabla \cdot \left( ( a^{*}\phi_{j}^{*} - \sigma_{j}^{*}  ) \nabla g \right)$. We see that $\frac{\partial F_{ij}^{o,1}}{\partial a}$ is given by the sum of two local terms plus the error term which is described through the solution of an auxiliary equation with r.h.s given in terms of $\nabla g$ (thus it would be of order $o(R^{-1})$). Next we plug (\ref{repr_form_1}) into the covariance estimate (\ref{cov_est}) which allows to estimate (\ref{correlation}) in terms of the derivative of $F_{ij}^{o,1}$ keeping the advantage of integrating against both $g$ and $g'$ so we could obtain a decay of order $\mathrm{o}\left( R^{-1}L^{-d}\right)$ (with a logarithmic correction). Main tools in the estimation of the r.h.s. of the covariance estimate will be the stochastic moment bounds for the correctors and the large-scale regularity theory.

In addition to the above estimate, we also study how the order of the decay is improved when we employ higher-order homogenization theory. More precisely, we consider the higher-order standard homogenization commutator
\begin{equs} \label{def_stand_comm}
\Xi_{ij}^{o,n}:= e_j \cdot (a-\bar{a}^1) (\nabla \phi _i^1+e_i) - \sum_{k=2}^n (-1)^{k-1} \bar{a}^{*,k}_{ji_1\dots i_{k-2}}e_{i_{k-1}} \cdot \partial ^{k-1}_{i_1 \dots i_{k-1}} \nabla \phi _i^1
\end{equs}
where, $1 \leq i,j \leq d$ and $1 < n \leq \tilde{d}$ for $\tilde{d}$ being the smallest integer larger than $\frac{d}{2}$ and $\bar{a}^{*,k}_{ji_1\dots i_{k-2}}$ the higher-order effective coefficients (see subsection \ref{higher-order} for precise definitions).  Moreover, we assume that we are in a Gaussian framework characterized by a covariance function with integrable decay of order $\mathrm{o}\left( |x|^{-d - \alpha_{0}} \right)$, for $0<\alpha_{0} \leq  \frac{d}{2}$ (see subsection \ref{ensemble} for a precise description of the class of ensembles we consider). We derive that the correlation has order $\mathrm{o}\left( R^{-d/2}L^{-d/2 -\alpha_{0}}\right)$ (up to a logarithmic correction) as well. That is, we see that the standard homogenization commutator inherits the property of $a$'s being weakly correlated, keeping the order of decay as well. The following is our main result.

\vspace{.5em}

\begin{thm}[Main Theorem] \label{mainthm}
If $d> 2$ is odd then
\begin{equs}P_{ijml}^{o,\frac{d+1}{2}}:=\Big \langle \int_{\R^d} g(x) \Xi_{ij}^{o,\frac{d+1}{2}}(x) \dx \int_{\R^d} g'(y) \Xi_{ml}^{o,\frac{d+1}{2}}(y) \dy \Big \rangle \lesssim R^{-\frac{d}{2}} L^{-\frac{d}{2}-\alpha_{0}} \ln \left( \frac{L}{R}\right).
\end{equs}
If $d\geq 2$ is even then
\begin{equs}
P_{ijml}^{o,\frac{d}{2}}:=\Big \langle \int_{\R^d} g(x) \Xi_{ij}^{o,\frac{d}{2}}(x) \dx \int_{\R^d} g'(y) \Xi_{ml}^{o,\frac{d}{2}}(y) \dy \Big \rangle \lesssim R^{-\frac{d}{2}}(\ln R)^{\frac{1}{2}} L^{-\frac{d}{2}-\alpha_{0}} \ln \left( \frac{L}{R}\right).
\end{equs}
Here $g$ and $g'$ are as described in (\ref{test-fcts}), $L>>R>>1$ and $\alpha_{0}$ is so that (\ref{c-decay}) holds. 
\end{thm}

\vspace{.5em}

\section{Preliminaries}

\subsection{Assumptions on the ensemble} \label{ensemble} We first describe the framework we adopt and the main ingredients we need for our analysis which hold true in this framework.

Let $\langle \cdot \rangle$ be a stationary and centered Gaussian ensemble of scalar fields $G$ on $\R^{d}$ characterized by its covariance function $c(x)=\langle G(x)G(0) \rangle$ which is assumed to satisfy the following 
\begin{equs} \label{c-decay}
|c(x)| \leq \frac{C_{0}}{(1+|x|)^{d+\alpha_{0}}}
\end{equs}
for some constants $0<\alpha_{0} \leq \frac{d}{2}$ and $C_{0}>0$. Moreover, we assume that the (always non-negative) Fourier transform of $c$ satisfies 
\begin{equs} \label{c-fourier}
\mathcal{F}c(k) \leq \frac{C_{1}}{(1+ |k|)^{d+2\alpha_{1}}}
\end{equs}
for some $\alpha_{1} \in (0,1)$ and some constant $C_{1}>0$. Next we identify $\langle \cdot \rangle$ with its push-forward under the map: $G \mapsto a$, where $a(x):=A(G(x))$, $A:\R \to \R^{d \times d}_{\lambda}$ a Lipschitz function and $\R^{d \times d}_{\lambda}$ the space of $\lambda$-elliptic matrices.

Note that in the framework adopted here, one can ensure that a spectral gap inequality (see for instance Lemma 3.1 in \cite{JO20}). Furthermore, Hellfer-Sj\"ostrand representation formula 
 \begin{equs} \label{HS}
 \cov[ F,H]  =\Big \langle  \int  \int  \frac{\partial F }{\partial G(x)} c(x-y) (1+\mathcal{L})^{-1} \frac{\partial H}{\partial G(y)} \dy \dx \Big \rangle
 \end{equs}
 holds  for every suitable random variables $F$ and $H$ (we refer the reader to section 4 in \cite{DO20} for precise statements and the definition of the differential operator $\mathcal{L}$ - here we only use the fact that this operator is bounded to get (\ref{cov_est})). We denote by $\frac{\partial F }{\partial G(x)}= \frac{\partial F }{\partial a(x)} A'(G(x))$, where the random tensor $\frac{\partial F_{ij}^{o,n}}{\partial a(x)}$ stands for the functional derivative of $F$ defined through
 \begin{equs} \label{def-frechet}
 \lim_{t \to 0} \frac{F(a+t \delta a) -F(a)}{t} = \int \frac{\partial F}{\partial a(x)} :  \delta a(x) \dx.
 \end{equs}
One of the main ingredients we use in this paper is the following covariance estimate (which is an immediate consequence of (\ref{HS}) and of the fact $\Big  \lvert \frac{\partial F }{\partial G(x)} \Big \rvert \lesssim \Big  \lvert \frac{\partial F }{\partial a(x)}  \Big \rvert$ )
 \begin{equs} \label{cov_est}
 \cov[ F,H]  \lesssim \int \Big \langle \Big  \lvert \frac{\partial F }{\partial a(x)}  \Big \rvert ^{2}  \Big \rangle ^{1/2} \int |c(x-y)| \Big \langle \Big |\frac{\partial H}{\partial a(y)}\Big |^{2}\Big \rangle ^{1/2} \dy \dx.
 \end{equs}

Finally, let us also mention that for the class of ensembles we consider here (in particular, because of (\ref{c-fourier})) we can show that realizations $G$ (thus $a$'s) are H\"older continuous with H\"older norms having bounded stochastic moments, that is,
\begin{equs} \label{a-holder}
\big \langle ||a||_{C^{\alpha'}(B_{1})}^{q} \big \rangle ^{1/q} \lesssim_{q,\alpha'} 1
\end{equs}
for any $0<\alpha'<\alpha_{1}$ and $q \geq 1$ (see Appendix A in \cite{JO20} for a proof).

Note that by $\lesssim$ we mean $\leq$ times a constant which depends only on $d, \lambda, \alpha_{0}, \alpha_{1}, ||A||_{C^{1}}$ and on quantities related to $c$. Moreover, note that we use the Einstein's summation convention.
 
 \vspace{.5em}
 
\subsection{Higher-order theory} \label{higher-order}

For reader's convenience let us first introduce the notions of higher-order correctors and effective coefficients and their main properties that we use in the following (see Definition 2.1 and Proposition 2.2 in \cite{DO20}).

\begin{defn} \label{correctors}
Let $\tilde{d}$ be the smallest integer larger than $\frac{d}{2}$. The correctors $(\phi^{n})_{0 \leq n \leq \tilde{d}}$, the flux correctors $(\sigma^{n})_{0 \leq n \leq \tilde{d}}$ and the effective coefficients $(\bar{a}^{n})_{1 \leq n \leq \tilde{d}}$ are inductively defined as follows
\begin{enumerate}
\item[$\bullet$] $\phi^{0}:=1$ and $\phi^{n}:= (\phi^{n}_{i_{1} \dots i_{n}})_{1\leq i_{1}, \dots,i_{n} \leq d}$ for any $1 \leq n \leq \tilde{d}$, with $\phi^{n}_{i_{1} \dots i_{n}}$ a scalar field satysfying
\begin{equs} \label{corr-eq}
-\nabla \cdot a \nabla \phi^{n}_{i_{1} \dots i_{n}}= \nabla \cdot \left( \left(a \phi^{n-1}_{i_{1} \dots i_{n-1}} - \sigma^{n-1}_{i_{1} \dots i_{n-1}}\right) e_{i_{n}} \right)
\end{equs}
with $\nabla \phi^{n}_{i_{1} \dots i_{n}}$ being stationary, centered and having finite second moments.
\item[$\bullet$] $\bar{a}^{n}:= (\bar{a}^{n}_{i_{1} \dots i_{n-1}})_{1\leq i_{1}, \dots,i_{n-1} \leq d}$ for any $1 \leq n \leq \tilde{d}$, with $\bar{a}^{n}_{i_{1} \dots i_{n-1}}$ the matrix given by
\begin{equs} \label{effectivecoeff}
\bar{a}^{n}_{i_{1}  \dots i_{n-1}} e_{i_{n}}:= \Big \langle a \left( \nabla \phi^{n}_{i_{1} \dots i_{n}}+ \phi^{n-1}_{i_{1} \dots i_{n-1}}e_{i_{n}}\right)  \Big \rangle.
\end{equs}
\item[$\bullet$] $\sigma^{0}:=0$ and $\sigma^{n}:= (\sigma^{n}_{i_{1} \dots i_{n}})_{1\leq i_{1}, \dots,i_{n} \leq d}$ for any $1 \leq n \leq \tilde{d}$, with $\sigma^{n}_{i_{1} \dots i_{n}}$ a skew-symmetric matrix satysfying
\begin{equs} \label{fluxcorr}
\nabla \cdot \sigma^{n}_{i_{1} \dots i_{n}} := a \nabla \phi^{n}_{i_{1} \dots i_{n}} + \left(a \phi^{n-1}_{i_{1} \dots i_{n-1}} - \sigma^{n-1}_{i_{1} \dots i_{n-1}}\right) e_{i_{n}} -\bar{a}^{n}_{i_{1}  \dots i_{n-1}} e_{i_{n}}
\end{equs}
with $\nabla \sigma^{n}_{i_{1} \dots i_{n}}$ being stationary, centered and having finite second moments. Here we denote by $(\nabla \cdot \sigma^{n}_{i_{1} \dots i_{n}})_{j}= \sum_{k=1}^{d} \partial_{k} (\sigma^{n}_{i_{1} \dots i_{n}})_{jk}$, $1 \leq j \leq d$.
\end{enumerate}
\end{defn}

\begin{prop} \label{momentbounds}
All the quantities in definition \ref{correctors} exist and they satisfy for all $1 \leq n \leq \tilde{d}$ and $p \geq 1$
\begin{equs} \label{momentboundsest}
|\alpha^{n}| \leq 1, \quad \big \langle |\nabla \phi^{n}|^{p} \big \rangle^{1/p}\lesssim_{p,n} 1, \quad \big \langle | \phi^{n}(x)|^{p} \big \rangle^{1/p} + \big \langle | \sigma^{n}(x)|^{p} \big \rangle^{1/p}\lesssim_{p,n} \mu_{d,n}(x)
\end{equs}
where
\begin{equs}
\mu_{d,n}(x) := \begin{cases} 1, &\quad \quad \text{ if } \ n < \tilde{d} \\
\ln^{1/2}(2+|x|), &\quad \quad \text{ if } \ n = \tilde{d} \ \text{ and } \ d \ \text{ even}\\
1+|x|^{1/2}, &\quad \quad \text{ if } \ n = \tilde{d} \ \text{ and } \ d \ \text{ odd}. \end{cases}
\end{equs}
\end{prop}

Next we explain why definition (\ref{def_stand_comm}) is reasonably derived from the definitions of higher-order commutators given in \cite{DO20}. The \textit{homogenization commutator} $\Xi^{1}[\nabla w] = (a-\bar{a}) \nabla w$ naturally extends to the higher-order as
\begin{equs} 
\Xi^{n}[\nabla w]:= (a- \sum_{k=1}^n \bar{a}^{k}_{i_1\dots i_{k-1}}  \partial ^{k-1}_{i_1 \dots i_{k-1}} ) \nabla w= (a-\bar{a}) \nabla w - \sum_{k=2}^n \bar{a}^{k}_{i_1\dots i_{k-1}}  \partial ^{k-1}_{i_1 \dots i_{k-1}} \nabla w.\\
\ \label{def_comm}
\end{equs}
Then the \textit{standard homogenization commutator} $\Xi^{o,n} [\nabla \bar{w}]$ is given by $\Xi^{n}$ applied to the nth-order Taylor polynomial of the nth-order two-scale expansion of $\bar{w}$. However a more explicit formula for $\Xi^{o,n} [\nabla \bar{w}]$ is available (see Lemma 3.5 in \cite{DO20}) and if this formula is applied to the linear functions $\bar{w} (x)=x_{i}$ we get
\begin{equs}
\Xi_{ij}^{o,n}= e_{j} \cdot  \Xi^{n} [\nabla \phi_{i}^{1}+e_{i}].
\end{equs}
Now for our analysis, especially when deriving representation formulas for the Malliavin derivatives, it is more convenient to work with (\ref{def_stand_comm})
which is defined through the transposes of the higher-order effective coefficients. For (\ref{def_stand_comm}), we use the following alternative representation of $\Xi^{n}$
\begin{equs} \label{def_comm_trans}
e_{j} \cdot \Xi^{n} [\nabla w] = a^{*} e_{j} \cdot \nabla w - \sum_{k=1}^n (-1)^{k-1}  \bar{a}^{*,k}_{ji_1\dots i_{k-2}}e_{i_{k-1}} \cdot \partial  ^{k-1}_{i_1 \dots i_{k-1}}\nabla w.
\end{equs}
The above is a consequence of Lemma 2.4 in \cite{DO20} which extends the fact $\bar{a}^{*,1}= (\bar{a}^{1})^{*}$ to the higher-order,
\begin{equs} \label{symmetries}
Sym_{i_{1} \dots i_{n}} (e_{j } \cdot \bar{a}^{n}_{i_1\dots i_{n-1}}  e_{{i_{n}}}) = (-1)^{n+1} Sym_{i_{1} \dots i_{n}} (e_{i_{n}} \cdot \bar{a}^{*,n}_{ji_1\dots i_{n-2}}  e_{{i_{n-1}}}) 
 \end{equs}
 where $Sym_{i_{1} \dots i_{k}} T_{i_{1} \dots i_{k}} := \frac{1}{k!} \sum _{\sigma \in S_{k}} T_{i_{\sigma (1)} \dots i_{\sigma (k)} }$, $T$ a kth-order tensor and $S_{k}$ the set of permutations of $\{1, \dots k \}$. 
 
 \vspace{.5em}
 
 \section{Proof of Main Theorem}
 
Since we intend to bound quantity (\ref{correlation}) via the covariance estimate (\ref{cov_est}), we first derive a suitable representation formula for the derivative of $F_{ij}^{o,n} := \int_{\R^d} g(x) \Xi_{ij}^{o,n}(x) \dx$. Our intention is to get as many derivatives as possible for the r.h.s of the equation that the term $\nabla h_{j}$ satisfies. We show the following

  \begin{prop} [Representation formula] \label{repr_form_prop}
 \begin{equs} \label{repr_form}
  \frac{\partial F_{ij}^{o,n}}{\partial a} = (\nabla \phi _{i}^{1}+e_{i}) \otimes  \left( \sum_{k=0}^{n} \partial^{k}_{i_{1} \dots i_{k}} g (\nabla \phi_{ji_{1} \dots i_{k}}^{*,k+1} +e_{i_{k}} \phi_{ji_{1} \dots i_{k-1}}^{*,k}) + \nabla h_{j} \right) 
 \end{equs}
 with $h_{j}$ solving
 \begin{equs} \label{eq-hj}
 - \nabla \cdot a^{*} \nabla h_{j} = \nabla \cdot \left( ( a^{*}\phi_{ji_{1} \dots i_{n-1}}^{*,n} - \sigma_{ji_{1} \dots i_{n-1}}^{*,n}  ) \nabla \partial^{n-1}_{i_{1} \dots i_{n-1}} g \right).
 \end{equs}
 Note that in the sum appears in (\ref{repr_form}), we use the convention that the $k=0$-term is just $\nabla \phi_{j}^{*,1}+e_{j}$, while the $k=n$-term is just  $\partial^{n}_{i_{1} \dots i_{n}} g \ e_{i_{n}} \phi_{ji_{1} \dots i_{n-1}}^{*,n}$ (see definition \ref{correctors}).

 \end{prop}
 
  \begin{proof}
 We have (integrating by parts)
  \begin{equs}
 {}&\frac{F_{ij}^{o,n}(a+t \delta a) -F_{ij}^{o,n}(a)}{t} \\
 &\quad =  \frac{1}{t}  \int g e_{j} \cdot (a+t \delta a-\bar{a}^1) (\nabla \phi _i^1(a+t \delta a)+e_i) \\ 
 &\quad \ \ -\frac{1}{t}  \int  \sum_{k=2}^n (-1)^{2(k-1)}\ \partial ^{k-1}_{i_1 \dots i_{k-1}}g \ \bar{a}^{*,k}_{ji_1\dots i_{k-2}}e_{i_{k-1}} \cdot  \nabla \phi _i^1(a+t \delta a) \\
 &\quad \ \ -\frac{1}{t}  \int  g e_j \cdot (a-\bar{a}^1) (\nabla \phi _i^{1}(a)+e_i)  \\
 &\quad \ \ +\frac{1}{t}  \int   \sum_{k=2}^n (-1)^{2(k-1)}\ \partial ^{k-1}_{i_1 \dots i_{k-1}}g \  \bar{a}^{*,k}_{ji_1\dots i_{k-2}}e_{i_{k-1}} \cdot \nabla \phi _i^{1}(a) \\
&\quad  = \int g e_{j} \cdot \delta a \left(e_{i} + \nabla \phi_{i}^{1}(a+ t \delta a)\right) +  \int g e_{j} \cdot (a- \bar{a}^{1}) \frac{ \nabla \phi_{i}^{1}(a+ t \delta a) -\nabla \phi_{i}^{1}(a) }{t} \\
&\quad \ \ -\int   \sum_{k=2}^n  \partial ^{k-1}_{i_1 \dots i_{k-1}}g \ \bar{a}^{*,k}_{ji_1\dots i_{k-2}}e_{i_{k-1}} \cdot  \frac{\nabla \phi _i^1(a+t \delta a) -\nabla \phi_{i}^{1}(a) }{t}
 \end{equs}
 where $\nabla \phi _i^1(a+t \delta a) -\nabla \phi_{i}^{1}(a) := \nabla \psi _{i}(a,t \delta a)$, with $- \nabla \cdot (a+t \delta a) \frac{\nabla  \psi_{i}}{t}= \nabla \cdot \delta a (\nabla \phi_{i}^{1}(a)+e_{1})$ (see section 3.4 in \cite{GNO19}). Thus letting $t \to 0$ we get  
 \begin{equs} 
  \lim_{t \to 0}\frac{F_{ij}^{o,n}(a+t \delta a) -F_{ij}^{o,n}(a)}{t} &= \int g e_{j} \cdot \delta a \left(e_{i} + \nabla \phi_{i}^{1}\right) + \int ge_{j} \cdot (a-\bar{a}^{1}) \nabla \delta \phi_{i}  \\ 
 &\ \ \  - \int \sum_{k=2}^n  \partial ^{k-1}_{i_1 \dots i_{k-1}}g \ \bar{a}^{*,k}_{ji_1\dots i_{k-2}}e_{i_{k-1}} \cdot \nabla \delta \phi _i \label{frechet-der}
 \end{equs}
 where 
 \begin{equs} \label{eq-for-delta phi}
 -\nabla \cdot a \nabla \delta \phi_{i} = \nabla \cdot \delta a (\nabla \phi_{i}^{1}+e_{i}).
 \end{equs}
 
Next we further analyze  the r.h.s of (\ref{frechet-der}) to get the desired representation formula. The main ingredient we use is the following relation between the correctors (see (\ref{fluxcorr}) in definition \ref{correctors})
\begin{equs}\label{corr-relation}
\left( a \phi^{k-1}_{i_{1}\dots i_{k-1}} - \sigma^{k-1}_{i_{1}\dots i_{k-1}} \right) e_{i_{k}}= -a \nabla \phi^{k}_{i_{1}\dots i_{k}} + \nabla \cdot \sigma^{k}_{i_{1}\dots i_{k}}  - \bar{a}^{k}_{i_{1}\dots i_{k-1}}e_{i_{k}}
\end{equs}
which for $k=1$ reduces to the well known $(a-\bar{a}^{1})e_{i}=-a \nabla \phi^{1}_{i} + \nabla \cdot \sigma^{1}_{i}$.

We show by induction on $n$ the following
\begin{equs} 
\lim_{t \to 0}\frac{F_{ij}^{o,n}(a+t \delta a) -F_{ij}^{o,n}(a)}{t} &= \int \sum_{k=0}^{n} \partial^{k}_{i_{1} \dots i_{k}} g (\nabla \phi_{ji_{1} \dots i_{k}}^{*,k+1} +e_{i_{k}} \phi_{ji_{1} \dots i_{k-1}}^{*,k}) \cdot \delta a (\nabla \phi _{i}^{1}+e_{i}) \\
&\ \ \ + \int \left( a^{*}\phi_{ji_{1} \dots i_{n-1}}^{*,n} - \sigma_{ji_{1} \dots i_{n-1}}^{*,n}  \right) \nabla \partial^{n-1}_{i_{1} \dots i_{n-1}} g \cdot \nabla \delta \phi_{i}.  \label{maineqofprop}
\end{equs}
Note that the above gives the result (via (\ref{def-frechet})). Indeed, testing equation (\ref{eq-hj}) with $ \delta \phi_{i}$ the second term of the r.h.s of (\ref{maineqofprop}) turns into $-\int a^{*}\nabla h_{j}\cdot \nabla \delta \phi_{i}=-\int \nabla h_{j}\cdot a \nabla \delta \phi_{i}=\int \nabla h_{j}\cdot  \delta a (\nabla \phi _{i}^{1}+e_{i})$. Where the last equality is obtained by testing equation (\ref{eq-for-delta phi}) with $h_{j}$. 

Now for the induction we start with $n=1$. In that case (\ref{frechet-der}) reduces to 
\begin{equs}
\lim_{t \to 0}\frac{F_{ij}^{o,1}(a+t \delta a) -F_{ij}^{o,1}(a)}{t} = \int g e_{j} \cdot \delta a \left(e_{i} + \nabla \phi_{i}^{1}\right) + \int ge_{j} \cdot (a-\bar{a}^{1}) \nabla \delta \phi_{i} 
\end{equs}
We work with the second term of the r.h.s.
\begin{equs}
\int ge_{j} \cdot (a-\bar{a}^{1}) \nabla \delta \phi_{i} &= \int g(a^{*}-\bar{a}^{*,1})e_{j} \cdot  \nabla \delta \phi_{i} = \int (-a^{*}g \nabla \phi^{*,1}_{j} + g\nabla \cdot \sigma^{*,1}_{j}) \cdot  \nabla \delta \phi_{i} \\
&= \int \left( -a^{*} \nabla (g \phi^{*,1}_{j} ) + a^{*}\phi^{*,1}_{j}  \nabla g - \sigma^{*,1}_{j} \nabla g \right) \cdot  \nabla \delta \phi_{i} 
\end{equs}
where the first two terms result from Leibniz rule and the last from the following property of $\sigma_{i}$
\begin{equs}
\nabla \cdot (g \nabla \cdot \sigma_{i}) = - \nabla \cdot ( \sigma_{i} \nabla g), \ \ \text{ for any smooth enough } \ g,
\end{equs}
which is an easy consequence of the skew-symmetry of $\sigma_{i}$. Thus we have
\begin{equs}
\lim_{t \to 0}\frac{F_{ij}^{o,1}(a+t \delta a) -F_{ij}^{o,1}(a)}{t} &= \int g e_{j} \cdot \delta a \left(e_{i} + \nabla \phi_{i}^{1}\right) + \int -a^{*} \nabla (g \phi^{*,1}_{j} ) \cdot  \nabla \delta \phi_{i} \\
&\ \ \ + \int \left( a^{*}\phi^{*,1}_{j}  - \sigma^{*,1}_{j} \right) \nabla g  \cdot  \nabla \delta \phi_{i}. 
\end{equs}
Note that testing equation (\ref{eq-for-delta phi}) with $g \phi^{*,1}_{j}$ we get $-\int a^{*} \nabla (g \phi^{*,1}_{j} ) \cdot  \nabla \delta \phi_{i} = \int  \nabla (g \phi^{*,1}_{j} ) \cdot   \delta a (\nabla \phi^{1}_{i} +e_{i})$. Then
\begin{equs}
\lim_{t \to 0}\frac{F_{ij}^{o,1}(a+t \delta a) -F_{ij}^{o,1}(a)}{t} &= \int \left(  \left( e_{j}+\nabla \phi_{j}^{*,1} \right)g+  \phi_{j}^{*,1} \nabla g\right)\cdot \delta a \left(e_{i} + \nabla \phi_{i}^{1}\right) \\
&\ \ \ + \int \left( a^{*}\phi^{*,1}_{j}  - \sigma^{*,1}_{j} \right) \nabla g  \cdot  \nabla \delta \phi_{i}
\end{equs}
which is exactly (\ref{maineqofprop}) for $n=1$. Next assume that (\ref{maineqofprop}) holds for $n-1$. We show that it is true for $n$. Indeed, by (\ref{frechet-der}) we see that
\begin{equs}
\lim_{t \to 0} \frac{F_{ij}^{o,n}(a+t \delta a) -F_{ij}^{o,n}(a)}{t} &=  \lim_{t \to 0}\frac{F_{ij}^{o,n-1}(a+t \delta a) -F_{ij}^{o,n-1}(a)}{t} \\
&\ \ \  - \int   \partial ^{n-1}_{i_1 \dots i_{n-1}}g \ \bar{a}^{*,n}_{ji_1\dots i_{n-2}}e_{i_{n-1}} \cdot \nabla \delta \phi _i \\
 &= \int \sum_{k=0}^{n-1} \partial^{k}_{i_{1} \dots i_{k}} g (\nabla \phi_{ji_{1} \dots i_{k}}^{*,k+1} +e_{i_{k}} \phi_{ji_{1} \dots i_{k-1}}^{*,k}) \cdot \delta a (\nabla \phi _{i}^{1}+e_{i}) \\
&\ \ \ + \int \left( a^{*}\phi_{ji_{1} \dots i_{n-2}}^{*,n-1} - \sigma_{ji_{1} \dots i_{n-2}}^{*,n-1}  \right)  \partial^{n-1}_{i_{1} \dots i_{n-1}} g \ e_{i_{n-1}}\cdot \nabla \delta \phi_{i} \\
&\ \ \  - \int   \partial ^{n-1}_{i_1 \dots i_{n-1}}g \ \bar{a}^{*,n}_{ji_1\dots i_{n-2}}e_{i_{n-1}} \cdot \nabla \delta \phi _i.
\end{equs}
We use again (\ref{corr-relation}) for the middle term
\begin{equs}
&\int \partial^{n-1}_{i_{1} \dots i_{n-1}} g  \left( a^{*}\phi_{ji_{1} \dots i_{n-2}}^{*,n-1} - \sigma_{ji_{1} \dots i_{n-2}}^{*,n-1}  \right)  \ e_{i_{n-1}}\cdot \nabla \delta \phi_{i} \\
&\ \ \ = \int \partial^{n-1}_{i_{1} \dots i_{n-1}} g  \left( -a^{*} \nabla \phi^{*,n}_{ji_{1}\dots i_{n-1}} + \nabla \cdot \sigma^{*,n}_{ji_{1}\dots i_{n-1}}  - \bar{a}^{*,n}_{ji_{1}\dots i_{n-2}}e_{i_{n-1}}  \right)  \cdot \nabla \delta \phi_{i}.
\end{equs}
Hence
\begin{equs} 
 \lim_{t \to 0}\frac{F_{ij}^{o,n}(a+t \delta a) -F_{ij}^{o,n}(a)}{t} &= \int \sum_{k=0}^{n-1} \partial^{k}_{i_{1} \dots i_{k}} g (\nabla \phi_{ji_{1} \dots i_{k}}^{*,k+1} +e_{i_{k}} \phi_{ji_{1} \dots i_{k-1}}^{*,k}) \cdot \delta a (\nabla \phi _{i}^{1}+e_{i}) \\
&\ \ \ - \int a^{*}\nabla (\partial^{n-1}_{i_{1} \dots i_{n-1}} g \   \phi^{*,n}_{ji_{1}\dots i_{n-1}} ) \cdot \nabla \delta \phi_{i}  \label{lasteqofprop} \\
&\ \ \ + \int \left ( a^{*}\phi^{*,n}_{ji_{1}\dots i_{n-1}} \nabla \partial^{n-1}_{i_{1} \dots i_{n-1}} g - \sigma^{*,n}_{ji_{1}\dots i_{n-1}}   \nabla \partial^{n-1}_{i_{1} \dots i_{n-1}} g   \right)\cdot \nabla \delta \phi_{i} 
\end{equs}
where we used again Leibniz rule and the skew-symmetry of $\sigma^{*,n}_{ji_{1}\dots i_{n-1}}$. For the middle term we use equation (\ref{eq-for-delta phi}) tested with $\partial^{n-1}_{i_{1} \dots i_{n-1}} g \   \phi^{*,n}_{ji_{1}\dots i_{n-1}}$ to get 
\begin{equs}
&- \int a^{*}\nabla (\partial^{n-1}_{i_{1} \dots i_{n-1}} g \   \phi^{*,n}_{ji_{1}\dots i_{n-1}} ) \cdot \nabla \delta \phi_{i} \\
&\ \ \ \ \ \ \ \ \ \ \ \ = \int \nabla (\partial^{n-1}_{i_{1} \dots i_{n-1}} g \   \phi^{*,n}_{ji_{1}\dots i_{n-1}} ) \cdot \delta a (\nabla \phi _{i}^{1}+e_{i}) \\ 
& \ \ \ \ \ \ \ \ \ \ \ \ =\int \left(  \partial^{n}_{i_{1} \dots i_{n}} g \ \phi^{*,n}_{ji_{1}\dots i_{n-1}} e_{i_{n}}+ \partial^{n-1}_{i_{1} \dots i_{n-1}} g \  \nabla   \phi^{*,n}_{ji_{1}\dots i_{n-1}} \right) \cdot \delta a (\nabla \phi _{i}^{1}+e_{i}). 
\end{equs} 
Substituting in (\ref{lasteqofprop}) and recalling that for the sum $\sum_{k=0}^{n-1} $ we have the convention that the $(n-1)$-term is just $ \partial^{n-1}_{i_{1} \dots i_{n-1}} g \ e_{i_{n-1}} \phi_{ji_{1} \dots i_{n-2}}^{*,n-1}$ we conclude the proof.
 \end{proof}
 
 \vspace{.5em}
 
 Next we study the solution $h_{j}$ of (\ref{eq-hj}) deriving some bounds that will be useful for the proof of the main theorem. In the sequel we assume that $d$ is odd and we denote by $\tilde{d}:=\frac{d+1}{2}$ (note that the proof when $d$ is even is similar - the slightly different bound comes from the stochastic moment bounds of the correctors). 
 
 \begin{lemma} \label{bounds-hj}
 Let $h_{j}$ be a solution of (\ref{eq-hj}) with $n=\tilde{d}$, then for any $|x| > \frac{L}{2}$ it holds
 \begin{equs}
 \Big \langle \big |  \nabla h_{j}(x)  \big |^{4} \Big \rangle ^{1/4} \lesssim R^{-d/2}L^{-d}. \label{bound-hj}
 \end{equs}
 \end{lemma}

\vspace{.5em}

 \begin{proof}
 Recall that $h_{j}$ satisfies the equation $- \nabla \cdot a^{*} \nabla h_{j} = \nabla \cdot f_{j}$, where $f_{j}:= ( a^{*}\phi_{ji_{1} \dots i_{\tilde{d}-1}}^{*,\tilde{d}} - \sigma_{ji_{1} \dots i_{\tilde{d}-1}}^{*,\tilde{d}}  ) \nabla \partial^{\tilde{d}-1}_{i_{1} \dots i_{\tilde{d}-1}} g \ $. Note that $\ \supp f_{j} \subset B_{R}$, in particular $h_{j}$ is $a^{*}-$harmonic in $B_{R}^{c}$ which contains $B_{L/2}^{c}$ (recall that $L>>R$). For any $|x| > \frac{L}{2}$, we bound first $\fint_{B_{\frac{L}{4}}(x)} |\nabla h_{j}|^{2}$ using the Lipschitz estimate (or mean-value property) of Theorem 1 in \cite{GNO20}. To achieve this we consider the auxiliary function $v$ which solves
\begin{equs}
-\nabla \cdot a \nabla v_{j} = \nabla \cdot \tilde{f}_{j}, \ \ \text{ where } \ \ \tilde{f}_{j}:=\mathcal{X}_{B^{c}_{L/4}}\nabla h_{j}.
\end{equs}
We then have
\begin{equs}
\int_{B_{\frac{L}{4}}^{c}} |\nabla h_{j}|^{2} = \int \tilde{f}_{j} \cdot \nabla h_{j} = \int a \nabla v_{j}\cdot \nabla h_{j} &= \int  \nabla v_{j} \cdot  f_{j} \\
&\lesssim \left( \int_{B_{R}} |\nabla v_{j}|^{2} \right)^{1/2} \left( \int_{B_{R}} |f_{j}|^{2} \right)^{1/2}.
\end{equs}
Denoting by $R^{*}_{x}= \max \{r^{*}(x), R\} $, where $r^{*}$ the minimal random radius of Theorem 1 in \cite{GNO20}, we estimate
\begin{equs}
\left( \int_{B_{R}} |\nabla v_{j}|^{2} \right)^{1/2} &\lesssim (R^{*}_{0})^{d/2} \left( \fint_{B_{R^{*}_{0}}} |\nabla v_{j}|^{2} \right)^{1/2}\\
&\lesssim (R^{*}_{0})^{d/2} \left( \fint_{B_{L/2}} |\nabla v_{j}|^{2} \right)^{1/2} \ \  \ \  (v \text{ is } a \text{-harmonic in } B_{L/2}) \\
&\lesssim \frac{(R^{*}_{0})^{d/2}}{L^{d/2}} \left( \int |\tilde{f}_{j}|^{2} \right)^{1/2}.  \ \ \ \ \ \ \ \ \ \  (\text{by the energy estimate})
\end{equs}
Thus we get
\begin{equs}
\int_{B_{\frac{L}{4}}^{c}} |\nabla h_{j}|^{2} \lesssim \frac{(R^{*}_{0})^{d/2}}{L^{d/2}} \left( \int_{B_{\frac{L}{4}}^{c}} |\nabla h_{j}|^{2} \right)^{1/2}\left( \int_{B_{R}} |f_{j}|^{2} \right)^{1/2}
\end{equs}
that is,
\begin{equs}
\left( \fint_{B_{L/4}(x)} |\nabla h_{j}|^{2} \right)^{1/2} \lesssim \frac{1}{L^{d/2}} \left( \int_{B_{L/4}^{c}} |\nabla h_{j}|^{2} \right)^{1/2} \lesssim \frac{(R^{*}_{0})^{d/2}}{L^{d}} \left( \int |f_{j}|^{2} \right)^{1/2}.
\end{equs}
Next we estimate $\Big \langle \big |  \nabla h_{j}(x)  \big |^{4} \Big \rangle^{1/4}$ for every $|x| > \frac{L}{2}$. By small-scale regularity we have, for $|x| > \frac{L}{2}$,
\begin{equs}
|\nabla h_{j} (x)| \lesssim C(a) \left( \fint_{B_{R/4}(x)} |\nabla h_{j}|^{2} \right)^{1/2},
\end{equs}
where $C(a)$ denotes the H\"older constant of the coefficient field $a$ (note that for $L>4R$, $B_{R/4}(x) \subset B_{L/4}^{c}$ and $h_{j}$ is $a^{*}$-harmonic in $B_{L/4}^{c}$). Then we apply Lipschitz estimate once more to get (assuming that $L>4R$)
\begin{equs}
|\nabla h_{j} (x)| \lesssim C(a) \frac{(R^{*}_{x})^{d/2}}{R^{d/2}} \left( \fint_{B_{R^{*}_{x}}(x)} |\nabla h_{j}|^{2} \right)^{1/2} &\lesssim C(a) \frac{(R^{*}_{x})^{d/2}}{R^{d/2}} \left( \fint_{B_{L/4}(x)} |\nabla h_{j}|^{2} \right)^{1/2} \\
&\lesssim C(a) \frac{(R^{*}_{x})^{d/2}}{R^{d/2}} \frac{(R^{*}_{0})^{d/2}}{L^{d}} \left( \int |f_{j}|^{2} \right)^{1/2}.
\end{equs}
Then we take expectation, we use the fact that both $C(a)$ (see (\ref{a-holder})) and $r^{*}(x)$ have uniformly bounded stochastic moments and choosing the worst scenario for the factor we get
\begin{equs}
\Big \langle \big |  \nabla h_{j}(x)  \big |^{4} \Big \rangle ^{1/4} \lesssim \frac{R^{d/2}}{L^{d}} \bigg \langle  \left( \int |f_{j}|^{2} \right)^{8} \bigg \rangle ^{1/16}. 
\end{equs}
It remains to estimate the r.h.s using Minkowski's integral inequality and Proposition \ref{momentbounds} (recall that $f_{j}:= ( a^{*}\phi_{ji_{1} \dots i_{\tilde{d}-1}}^{*,\tilde{d}} - \sigma_{ji_{1} \dots i_{\tilde{d}-1}}^{*,\tilde{d}}  ) \nabla \partial^{n-1}_{i_{1} \dots i_{n-1}} g \ $). We have
\begin{equs}
\Big \langle \big |  \nabla h_{j}(x)  \big |^{4} \Big \rangle ^{1/4} &\lesssim \frac{R^{d/2}}{L^{d}}  \left( \int \langle |f_{j}|^{16}\rangle^{1/8} \right)^{1/2} \\
&\lesssim \frac{R^{d/2}}{L^{d}} \left( \int_{B_{R}} |D^{\tilde{d}} g|^{2} \langle |a^{*}\phi_{ji_{1} \dots i_{\tilde{d}-1}}^{*,\tilde{d}} - \sigma_{ji_{1} \dots i_{\tilde{d}-1}}^{*,\tilde{d}} |^{16}\rangle^{1/8} \right)^{1/2} \\
&\lesssim \frac{R^{d/2}}{L^{d}}  \left( \int_{B_{R}} |D^{\tilde{d}} g|^{2} |z| \right)^{1/2} \lesssim \frac{R^{d/2}} {L^{d}} R^{-\frac{d}{2}-\tilde{d}}R^{\frac{1}{2}} = R^{-d/2}L^{-d}.
\end{equs}

 \end{proof}
 
 \vspace{.5em}
 
\begin{remark} \label{bounds-hj+}
 Note that by translation we easily see that the solution $h_{l}'$ of 
\begin{equs}
  - \nabla \cdot a^{*} \nabla h_{l}' = \nabla \cdot \left( ( a^{*}\phi_{li_{1} \dots i_{\tilde{d}-1}}^{*,\tilde{d}} - \sigma_{li_{1} \dots i_{\tilde{d}-1}}^{*,\tilde{d}}  ) \nabla \partial^{\tilde{d}-1}_{i_{1} \dots i_{\tilde{d}-1}} g' \right)
 \end{equs}
 satisfies, for any $|x-Le| > \frac{L}{2}$,
 \begin{equs} \label{bound-hl}
 \Big \langle \big |  \nabla h'_{l}(x)  \big |^{4} \Big \rangle ^{1/4} \lesssim R^{-d/2}L^{-d}. 
 \end{equs}
Note also that the above bounds can be rephrased (assuming that $L>4R$) as
\begin{equs}
{}&\Big \langle \big |  \nabla h_{j}(x)  \big |^{4} \Big \rangle ^{1/4} \lesssim R^{-d/2}|x|^{-d}, \ \text{ for any } \ |x| > 2R \label{bound-hj2} \\
& \Big \langle \big |  \nabla h'_{l}(x)  \big |^{4} \Big \rangle ^{1/4} \lesssim R^{-d/2}|x-Le|^{-d}, \ \text{ for any }   \ |x-Le| > 2R. \label{bound-hl2}
\end{equs}
A last observation that will be useful in the sequel is the following 
\begin{equs} \label{CZ}
\left( \int   \Big \langle \big |  \nabla h_{j}  \big |^{4} \Big \rangle^{2/4} \right)^{1/2} \lesssim R^{-d}
\end{equs}
which is a consequence of annealed Calderon-Zygmund estimates (see Proposition 7.1 in \cite{JO20}) and Proposition \ref{momentbounds}
\begin{equs} 
\left( \int   \Big \langle \big |  \nabla h_{j}  \big |^{4} \Big \rangle^{2/4} \right)^{1/2} \lesssim \left( \int   \Big \langle \big |   f_{j}  \big |^{4} \Big \rangle^{2/4} \right)^{1/2} \lesssim R^{1/2} \left( \int    \big |  D^{\tilde{d}}g \big |^{2} \right)^{1/2} \lesssim R^{1/2} R^{-\frac{d}{2}-\tilde{d}} = R^{-d}.
\end{equs}
A similar bound holds for $h_{l}'$ as well.
 \end{remark}
 
 \vspace{.5em}
 
We are now ready to prove our main theorem.

 \begin{proof}[Proof of Theorem \ref{mainthm}]
 Combining estimate (\ref{cov_est}) with Proposition \ref{repr_form_prop} we get
 \begin{equs} 
P_{ijml}^{o,\tilde{d}} &\lesssim \int \ \left[ \Big \langle \big |  (\nabla \phi _{i}^{1}(x)+e_{i}) \otimes g(x)\ (\nabla \phi_{j}^{*,1}(x)+e_{j})   \big |^{2} \Big \rangle^{1/2}\right.\\
&\ \ \ \ \ \ \ \ \ + \Big \langle \Big |  (\nabla \phi _{i}^{1}(x)+e_{i}) \otimes \sum_{k=1}^{\tilde{d}-1} \partial^{k}_{i_{1} \dots i_{k}} g(x) (\nabla \phi_{ji_{1} \dots i_{k}}^{*,k+1} (x)+e_{i_{k}} \phi_{ji_{1} \dots i_{k-1}}^{*,k}(x))   \Big |^{2} \Big \rangle^{1/2}  \\
&\ \ \ \ \ \ \ \ \ + \Big \langle \big |  (\nabla \phi _{i}^{1}(x)+e_{i}) \otimes \partial^{\tilde{d}}_{i_{1} \dots i_{\tilde{d}}} g(x) \ e_{i_{\tilde{d}}}\phi_{ji_{1} \dots i_{\tilde{d}-1}}^{*,\tilde{d}}(x)  \big |^{2} \Big \rangle^{1/2} \label{mainthm1} \\
&\ \ \ \ \ \ \ \ \ +\left. \Big \langle \big |  (\nabla \phi _{i}^{1}(x)+e_{i}) \otimes \nabla h_{j} (x) \big |^{2} \Big \rangle^{1/2}  \right] \\
&\ \  \ \ \ \times \int |c(x-y)| \bigg [  \text{ same terms with } i \leftrightarrow m, j \leftrightarrow l,x \leftrightarrow  y  \text{ and } g \leftrightarrow g' \  \bigg ] \dy \dx.
 \end{equs}

Using Proposition \ref{momentbounds} we may estimate
\begin{equs}
&\bullet \ \  \Big \langle \Big |  (\nabla \phi _{i}^{1}+e_{i}) \otimes g (\nabla \phi_{j}^{*,1}+e_{j})   \Big |^{2} \Big \rangle^{1/2} \lesssim \ |g| \Big \langle \big |  \nabla \phi _{i}^{1}+e_{i} \big |^{4} \Big \rangle^{1/4} \Big \langle \big |  \nabla \phi_{j}^{*,1}+e_{j})   \big |^{4} \Big \rangle^{1/4} \lesssim \ |g|. \\
&\bullet \ \ \Big \langle \Big |  (\nabla \phi _{i}^{1}+e_{i}) \otimes \sum_{k=1}^{\tilde{d}-1} \partial^{k}_{i_{1} \dots i_{k}} g (\nabla \phi_{ji_{1} \dots i_{k}}^{*,k+1} +e_{i_{k}} \phi_{ji_{1} \dots i_{k-1}}^{*,k})   \Big |^{2} \Big \rangle^{1/2} \\
& \ \ \ \ \ \ \ \ \ \lesssim \Big \langle \big |  \nabla \phi _{i}^{1}+e_{i} \big |^{4} \Big \rangle^{1/4} \sum_{k=1} ^{\tilde{d}-1} \big | \partial^{k}_{i_{1} \dots i_{k}} g \big | \left(  \Big \langle \big | \nabla \phi_{ji_{1} \dots i_{k}}^{*,k+1}  \big |^{4} \Big \rangle^{1/4} +  \Big \langle \big | \phi_{ji_{1} \dots i_{k-1}}^{*,k}  \big |^{4} \Big \rangle^{1/4} \right) \\
& \ \ \ \ \ \ \ \ \ \lesssim \sum_{k=1} ^{\tilde{d}-1} \big | \partial^{k}_{i_{1} \dots i_{k}} g \big |. \ \ \ \ \ \ \ \ (\text { note that } k<\tilde{d}\ ) \\
&\bullet \ \ \Big \langle \big |  (\nabla \phi _{i}^{1}+e_{i}) \otimes \partial^{\tilde{d}}_{i_{1} \dots i_{\tilde{d}}} g \ \phi_{ji_{1} \dots i_{\tilde{d}-1}}^{*,\tilde{d}}  \big |^{2} \Big \rangle^{1/2} \\
& \ \ \ \ \ \ \ \ \ \lesssim \ \big |  \partial^{\tilde{d}}_{i_{1} \dots i_{\tilde{d}}} g \big | \Big \langle \big |  \nabla \phi _{i}^{1}+e_{i}  \big |^{4} \Big \rangle^{1/4} \Big \langle \big | \phi_{ji_{1} \dots i_{\tilde{d}-1}}^{*,\tilde{d}}  \big |^{4} \Big \rangle^{1/4} \lesssim \ |z|^{1/2} \ \big |  \partial^{\tilde{d}}_{i_{1} \dots i_{\tilde{d}}} g \big |. \\
&\bullet \ \ \Big \langle \big |  (\nabla \phi _{i}^{1}+e_{i}) \otimes \nabla h_{j}  \big |^{2} \Big \rangle^{1/2} \lesssim \Big \langle \big |  \nabla \phi _{i}^{1}+e_{i}  \big |^{4} \Big \rangle^{1/4} \Big \langle \big |  \nabla h_{j}  \big |^{4} \Big \rangle^{1/4} \lesssim \Big \langle \big |  \nabla h_{j}  \big |^{4} \Big \rangle^{1/4}.
\end{equs}

Now we return to (\ref{mainthm1}), we apply the above estimates and multiply to get
\begin{equs}
\ \ P_{ijml}^{o,\tilde{d}} &\lesssim  \int \bigg ( \sum_{k=0} ^{\tilde{d}-1} \big | \partial^{k}_{i_{1} \dots i_{k}} g (x)\big | + |x|^{\frac{1}{2}} \ \big |  \partial^{\tilde{d}}_{i_{1} \dots i_{\tilde{d}}} g(x) \big | \bigg ) \\
&\quad \quad \times \int |c(x-y)| \bigg ( \sum_{k=0} ^{\tilde{d}-1} \big | \partial^{k}_{i_{1} \dots i_{k}} g'(y) \big | + |y-Le|^{\frac{1}{2}} \ \big |  \partial^{\tilde{d}}_{i_{1} \dots i_{\tilde{d}}} g'(y) \big | \bigg ) \dy \dx \\
&\  + \int \Big \langle \big |  \nabla h_{j}(x)  \big |^{4} \Big \rangle^{\frac{1}{4}} \int |c(x-y)| \bigg ( \sum_{k=0} ^{\tilde{d}-1} \big | \partial^{k}_{i_{1} \dots i_{k}} g'(y) \big | + |y-Le|^{\frac{1}{2}} \ \big |  \partial^{\tilde{d}}_{i_{1} \dots i_{\tilde{d}}} g'(y) \big | \bigg )   \dy \dx\\
&\ + \int  \bigg ( \sum_{k=0} ^{\tilde{d}-1} \big | \partial^{k}_{i_{1} \dots i_{k}} g(x) \big | + |x|^{\frac{1}{2}} \ \big |  \partial^{\tilde{d}}_{i_{1} \dots i_{\tilde{d}}} g(x) \big | \bigg ) \int |c(x-y)|  \Big \langle \big |  \nabla h'_{l}(y)  \big |^{4} \Big \rangle^{\frac{1}{4}} \dy \dx\\
&\ + \int \Big \langle \big |  \nabla h_{j}(x)  \big |^{4} \Big \rangle^{\frac{1}{4}}  \int |c(x-y)| \Big \langle \big |  \nabla h'_{l}(y)  \big |^{4} \Big \rangle^{\frac{1}{4}} \dy \dx.
\end{equs}

Next we focus on bounding each of these four terms. Starting with the first one we observe that it is enough to estimate the term
\begin{equs}
\int  |  g (x) |  \int |c(x-y)| \ | g'(y)  | \dy \dx \lesssim L^{-d-\alpha_{0}} \lesssim R^{-d/2}L^{-d/2-\alpha_{0}}
\end{equs}
where we used that $x \in B_{R}$ and $y \in B_{R}(Le)$ because of the supports of $g$ and $g'$ respectively. Then $|x-y| \geq L/2$ which gives the estimate if we use the integrable decay (\ref{c-decay}) we have assumed for $c$.

We proceed with the second term where we see once again that it is enough to estimate the ''subterm'' with the ''worst'' behaviour, that is
\begin{equs}
\int \Big \langle \big |  \nabla h_{j}(x)  \big |^{4} \Big \rangle^{\frac{1}{4}} \int |c(x-y)| |g'(y) |  \dy \dx.
\end{equs}
For, we split the domain of integration into two in order to be able to apply the estimate of Lemma \ref{bounds-hj} to $h_{j}$. So for the first term we apply bound (\ref{bound-hj}) and Young's convolution inequality, while for the second term we use estimate (\ref{CZ}) and the integrable decay (\ref{c-decay}) of $c$ together with Minkowski's integral inequality,
\begin{equs} 
{} &\int \Big \langle \big |  \nabla h_{j}(x)  \big |^{4} \Big \rangle^{\frac{1}{4}} \int |c(x-y)| |g'(y) |  \dy \dx \\
&\quad \lesssim R^{-d/2}L^{-d} \int_{B^{c}_{L/2}}  \int |c(x-y)| |g'(y) |  \dy \dx \\
&\quad \ \ +  \int_{B_{L/2}}  \int  \frac{|g'(y) |}{(1+|x-y|)^{d+\alpha_{0}}}  \Big \langle \big |  \nabla h_{j}(x)  \big |^{4} \Big \rangle^{\frac{1}{4}} \dy \dx \\
&\quad \lesssim R^{-d/2}L^{-d} ||g'||_{L^{1}}||c||_{L^{1}} \\
&\quad \ \ + \left( \int_{B_{L/2}} \left(  \int_{B_{R}(Le)}  \frac{|g'(y) |}{(1+|x-y|)^{d+\alpha_{0}}} \right)^{2} \dy \dx \right)^{1/2} \left( \int  \Big \langle \big |  \nabla h_{j}(x)  \big |^{4} \Big \rangle^{\frac{2}{4}} \dx \right)^{1/2} \\
&\quad \lesssim R^{-d/2}L^{-d}  + R^{-d} \int_{B_{R}(Le)} |g'(y) |   \left( \int_{B_{L/2}}  \frac{1}{(1+|x-y|)^{2d+2\alpha_{0}}} \dx \right)^{1/2} \dy \\
&\quad \lesssim R^{-d/2}L^{-d}  + R^{-d} L^{d/2} L^{-d-\alpha_{0}} \lesssim R^{-d/2}L^{-d/2 -\alpha_{0}}.
\end{equs}
It remains to bound the fourth term which is the most challenging. In order to be able to use the bounds of Lemma \ref{bounds-hj} and Remark \ref{bounds-hj+} we divide the domains of integration as follows
\begin{equs}
{}&\int \Big \langle \big |  \nabla h_{j}(x)  \big |^{4} \Big \rangle^{\frac{1}{4}}  \int |c(x-y)| \Big \langle \big |  \nabla h'_{l}(y)  \big |^{4} \Big \rangle^{\frac{1}{4}} \dy \dx \\
&\quad \lesssim R^{-d/2}L^{-d} \int_{B_{3R}(Le)} \left( \int |c(x-y)|^{2} \dy \right)^{1/2} \left( \int \Big \langle \big |  \nabla h'_{l}(y)  \big |^{4} \Big \rangle^{\frac{2}{4}} \dy \right)^{1/2}\dx \\
&\quad \ \ + R^{-d/2}L^{-d} \left( \int_{B^{c}_{3R}(Le)} \Big \langle \big |  \nabla h_{j}(x)  \big |^{4} \Big \rangle^{\frac{2}{4}} \dx \right)^{1/2} \left( \int_{B^{c}_{3R}(Le)}  \left( \int_{B_{3R}} |c(x-y)| dy \right)^{2} \dx \right)^{1/2} \\
&\quad \ \ + \int_{B^{c}_{3R}(Le)} \Big \langle \big |  \nabla h_{j}(x)  \big |^{4} \Big \rangle^{\frac{1}{4}}  \int_{B^{c}_{3R}} |c(x-y)| \Big \langle \big |  \nabla h'_{l}(y)  \big |^{4} \Big \rangle^{\frac{1}{4}} \dy \dx \\
&\quad \lesssim R^{-d/2}L^{-d} + \int_{B^{c}_{3R}(Le)} \Big \langle \big |  \nabla h_{j}(x)  \big |^{4} \Big \rangle^{\frac{1}{4}}  \int_{B^{c}_{3R}} |c(x-y)| \Big \langle \big |  \nabla h'_{l}(y)  \big |^{4} \Big \rangle^{\frac{1}{4}} \dy \dx
\end{equs}
using estimates (\ref{bound-hj}), (\ref{bound-hl}) and (\ref{CZ}). Now for the last term we need to further divide the domains. Precisely, we split the $x$-integral into $B_{3R}$ and $B^{c}_{3R}$ and the $y$-integral into $B_{3R}(Le)$ and $B^{c}_{3R}(Le)$. This produces four new terms that we estimate in the following. For the first one we may apply estimates (\ref{bound-hj2}) and (\ref{bound-hl2}). Then we divide the domain of $y$-integral once again and use the integrable decay (\ref{c-decay}) of $c$
\begin{equs}
 {}&\int_{B^{c}_{3R}(Le) \cap B^{c}_{3R}} \Big \langle \big |  \nabla h_{j}(x)  \big |^{4} \Big \rangle^{\frac{1}{4}}  \int_{B^{c}_{3R}(Le) \cap B^{c}_{3R}  } |c(x-y)| \Big \langle \big |  \nabla h'_{l}(y)  \big |^{4} \Big \rangle^{\frac{1}{4}} \dy \dx \\
 &\quad  \lesssim R^{-d} \int_{B^{c}_{3R}(Le) \cap B^{c}_{3R}}  \int_{B^{c}_{3R}(Le) \cap B^{c}_{3R}  } |x|^{-d} |y-Le|^{-d}|c(x-y)| \dy \dx \\
 &\quad  \lesssim R^{-d} \int_{B^{c}_{3R}(Le) \cap B^{c}_{3R}} |x|^{-d}   \int_{B^{c}_{3R}(Le) \cap B^{c}_{3R}  \cap \{|y-Le| > \frac{|x-Le|}{2}\}} |y-Le|^{-d}|c(x-y)|  \dy \dx \\
 &\quad \ \ + R^{-d} \int_{B^{c}_{3R}(Le) \cap B^{c}_{3R}} |x|^{-d}   \int_{B^{c}_{3R}(Le) \cap B^{c}_{3R}  \cap \{|y-Le| \leq \frac{|x-Le|}{2}\}}\frac{|y-Le|^{-d}}{(1+|x-y|)^{d+\alpha_{0}}} \dy \dx \\
&\quad  \lesssim R^{-d} \int_{B^{c}_{3R}(Le) \cap B^{c}_{3R}} |x|^{-d}  |x-Le|^{-d}  ||c||_{L^{1}} \dx \\
&\quad \ \ + R^{-d} \int_{B^{c}_{3R}(Le) \cap B^{c}_{3R}} |x|^{-d}   \int_{\{ 3R \leq |y-Le| \leq \frac{|x-Le|}{2}\}}\frac{|y-Le|^{-d}}{|x-Le|^{d+\alpha_{0}}} \dy \dx \\
&\quad \lesssim R^{-d} \int_{B^{c}_{3R}(Le) \cap B^{c}_{3R}} |x|^{-d}  |x-Le|^{-d}  \dx \\
&\quad \ \ + R^{-d} \int_{B^{c}_{3R}(Le) \cap B^{c}_{3R}} |x|^{-d}  \ln \frac{|x-Le|}{6R} |x-Le|^{-d-\alpha_{0}} \dx \\
&\quad  \lesssim R^{-d} \int_{B^{c}_{3R}(Le) \cap B^{c}_{3R}} |x|^{-d}  |x-Le|^{-d}  \dx + R^{-d} \int_{B^{c}_{3R}(Le) \cap B^{c}_{3R}} |x|^{-d}  |x-Le|^{-d-\alpha_{0}+1}  \dx.
\end{equs}
We then calculate
\begin{equs}
{}& R^{-d} \int_{B^{c}_{3R}(Le) \cap B^{c}_{3R}} |x|^{-d}  |x-Le|^{-d}  \dx\\
&\quad =2 R^{-d}\int_{\{|x|<|x-Le| \} \cap B_{3R}^{c}} |x|^{-d}    |x-Le|^{-d} \dx \\
&\quad  \lesssim R^{-d} \int_{\{|x|<|x-Le| \} \cap B_{3L}^{c}} |x|^{-2d}   \dx + R^{-d}L^{-d} \int_{\{|x|<|x-Le| \} \cap (B_{3L} \setminus B_{3R})} |x|^{-d}   \dx \\
& \quad \lesssim R^{-d} L^{-d}\ln \frac{L}{R} \lesssim R^{-d/2} L^{-d/2-\alpha_{0}}\ln \frac{L}{R}.
\end{equs}
Similarly $ \int_{B^{c}_{3R}(Le) \cap B^{c}_{3R}} |x|^{-d}  |x-Le|^{-d-\alpha_{0}+1}  \dx \lesssim R^{-d} L^{-d-\alpha_{0}+1}\ln \frac{L}{R}$ (note that in that term we could get rid of the logarithmic correction when $d \geq 4$).

For the next term we split the $y$-integral into $|x-y| > L/4$ and $|x-y| \leq L/4$. So in the first case we use the integrable decay (\ref{c-decay}) of $c$ to gain the power we need on $L$ together with estimate (\ref{CZ}). In the second case we see that $|x| \geq L/2$ which allows to apply Lemma \ref{bounds-hj} which we then combine with (\ref{CZ}). Indeed,
\begin{equs}
{}&\int_{B^{c}_{3R}(Le) \cap B^{c}_{3R}} \Big \langle \big |  \nabla h_{j}(x)  \big |^{4} \Big \rangle^{\frac{1}{4}}  \int_{B_{3R}(Le) \cap B^{c}_{3R}  } |c(x-y)| \Big \langle \big |  \nabla h'_{l}(y)  \big |^{4} \Big \rangle^{\frac{1}{4}} \dy \dx \\ 
&\quad  \lesssim \left( \int \Big \langle \big |  \nabla h_{j}(x)  \big |^{4} \Big \rangle^{\frac{2}{4}} \dx \right)^{1/2} \\
&\quad \quad \quad \quad \times \left( \int_{B^{c}_{3R}(Le) \cap B^{c}_{3R}} \left(  \int_{B_{3R}(Le) \cap \{|x-y| > L/4 \}  } |c(x-y)| \Big \langle \big |  \nabla h'_{l}(y)  \big |^{4} \Big \rangle^{\frac{1}{4}} dy\right)^{2} \dx \right)^{1/2} \\
&\quad \ \ + \int_{B^{c}_{3R}(Le) \cap B^{c}_{3R}} \Big \langle \big |  \nabla h_{j}(x)  \big |^{4} \Big \rangle^{\frac{1}{4}}  \int_{B_{3R}(Le) \cap \{|x-y| \leq L/4\}  } |c(x-y)| \Big \langle \big |  \nabla h'_{l}(y)  \big |^{4} \Big \rangle^{\frac{1}{4}} \dy \dx \\
&\quad  \lesssim R^{-d} \int_{B_{3R}(Le)  } \Big \langle \big |  \nabla h'_{l}(y)  \big |^{4} \Big \rangle^{\frac{1}{4}} \left(  \int_{B^{c}_{3R}(Le) \cap B^{c}_{3R} \cap \{|x-y| > L/4\}  } |c(x-y)|^{2} \dx \right)^{1/2} \dy \\
&\quad \ \ + R^{-d/2}L^{-d}  \int_{B^{c}_{3R}(Le) \cap B^{c}_{3R}}  \int_{B_{3R}(Le) }  |c(x-y)| \Big \langle \big |  \nabla h'_{l}(y)  \big |^{4} \Big \rangle^{\frac{1}{4}}  \dy \dx  \\
&\quad \lesssim R^{-d} R^{d/2}R^{-d} L^{-d/2 -\alpha_{0}} +R^{-d/2}L^{-d} R^{d/2}R^{-d} ||c||_{L^{1}} \lesssim R^{-d/2}L^{-d/2-\alpha_{0}}.
\end{equs}
Note that the term
\begin{equs}
 \int_{B^{c}_{3R}(Le) \cap B_{3R}} \Big \langle \big |  \nabla h_{j}(x)  \big |^{4} \Big \rangle^{\frac{1}{4}}  \int_{B^{c}_{3R}(Le) \cap B^{c}_{3R}  } |c(x-y)| \Big \langle \big |  \nabla h'_{l}(y)  \big |^{4} \Big \rangle^{\frac{1}{4}} \dy \dx 
\end{equs}
can be treated analogously. Finally, it remains to bound
\begin{equs}
 {}&\int_{B^{c}_{3R}(Le) \cap B_{3R}} \Big \langle \big |  \nabla h_{j}(x)  \big |^{4} \Big \rangle^{\frac{1}{4}}  \int_{B_{3R}(Le) \cap B^{c}_{3R}  } |c(x-y)| \Big \langle \big |  \nabla h'_{l}(y)  \big |^{4} \Big \rangle^{\frac{1}{4}} \dy \dx \\
 &\quad  \lesssim \left( \int \Big \langle \big |  \nabla h_{j}(x)  \big |^{4} \Big \rangle^{\frac{2}{4}} \dx \right)^{1/2} \left( \int_{B_{3R}} \left( \int_{B_{3R(Le)}} |c(x-y)| \Big \langle \big |  \nabla h'_{l}(y)  \big |^{4} \Big \rangle^{\frac{1}{4}} dy\right)^{2} \dx  \right)^{1/2} \\
 &\quad  \lesssim R^{-d} R^{d/2} L^{-d-\alpha_{0}} R^{d/2}  \left( \int \Big \langle \big |  \nabla h_{l}'(y)  \big |^{4} \Big \rangle^{\frac{2}{4}} dy \right)^{1/2} \lesssim R^{-d/2}L^{-d/2 -\alpha_{0}}
\end{equs}
where we used estimate (\ref{CZ}), the integrable decay (\ref{c-decay}) of $c$ together with the fact that $|x-y| \gtrsim L$ when $x \in B_{3R}$ and $y \in B_{3R}(Le)$.
 \end{proof}

 \vspace{2em}

\section*{Acknowledgements}
I would like to express my special thanks to Felix Otto for pointing out this open question, as well as for the fruitful discussions we had on the topic.
 
 \vspace{2em}

\bibliographystyle{plain}   
\bibliography{biblio}             
\index{Bibliography@\emph{Bibliography}}%
%
%

%

\end{document}